\documentclass[12pt]{article}

\usepackage{diagbox}
\usepackage{mathtools}
\usepackage{bbm}
\usepackage{latexsym}
\usepackage{epsfig}
\usepackage{amsmath,amsthm,amssymb,enumerate}

\usepackage[a-1b]{pdfx}
\usepackage{hyperref}
\usepackage[normalem]{ulem}

\usepackage{color}

\def\cH{{\mathcal H}}
\def\nn{\nonumber}
\def\a{\alpha} \def\b{\beta}  
\def\e{\varepsilon} \def\f{\phi}   \def\g{\gamma}
\def\G{\Gamma} \def\i{\iota} \def\k{\kappa}
     
 \def\m{\mu}  \def\p{\pi}
  \def\s{\sigma} 
 \def\om{\omega}

\def\va{\boldsymbol \a}
\def\vb{\boldsymbol \b}

\newtheorem{theorem}{Theorem}
\newtheorem{lemma}[theorem]{Lemma}

\newtheorem{Remark}{Remark}

\newcommand{\wh}[1]{\widehat{#1}}

\newcommand{\brac}[1]{\left(#1\right)}

\newcommand{\bfrac}[2]{\left(\frac{#1}{#2}\right)}

\newcommand{\set}[1]{\left\{#1\right\}}

\def\E{\mathbb{E}}

\def\Pr{\mathbb{P}}

\newcommand{\ignore}[1]{}

\def\cG{{\mathcal G}}
\def\cH{{\mathcal H}}

\newcommand{\card}[1]{\left|#1\right|}

\newcommand{\beq}[2]{\begin{equation}\label{#1}#2\end{equation}}
\newcommand{\mults}[1]{\begin{multline*}#1\end{multline*}}

\def\nn{\nonumber}

\def\cG{\mathcal{G}}

\def\bc{{\bf c}}
\newcommand{\bi}{\boldsymbol\sigma}

\usepackage{tikz}
\usetikzlibrary{decorations.pathmorphing}
\usetikzlibrary{positioning}
\usetikzlibrary{arrows,automata}
\usetikzlibrary{shapes.misc}
\usetikzlibrary{backgrounds}
\usetikzlibrary{arrows,shapes}

\newcommand{\upp}[1]{\langle#1\rangle}

\begin{document}
\author{Alan Frieze\thanks{Research supported in part by NSF grant DMS1952285} and Wesley Pegden\thanks{Research supported in part by NSF grant DMS1700365}\\Department of Mathematical Sciences\\Carnegie Mellon University\\Pittsburgh PA 15213}

\date{}
\title{Sequentially constrained Hamilton cycles in random graphs}
\maketitle
\begin{abstract}
We discuss the existence of Hamilton cycles in the random graph $G_{n,p}$ where there are restrictions caused by (i) coloring sequences, (ii) a subset of vertices must occur in a specific order and (iii) there is a bound on the number of inversions in the associated permutation.
\end{abstract}
\section{Introduction}
\subsection{Randomly colored random graphs}
In this paper we consider several questions related to Hamilton cycles in random graphs. Our first set of questions arise from randomly coloring the edges or vertices. Suppose we are given a graph $G=(V,E)$, $k$ colors $1,2,\ldots,k=O(1)$ and a map $c:E\to [k]$. A color pattern will be a sequence $\bc=(c_1,c_2,\ldots,c_n)$. Our first result concerns edge colored copies of $G_{n,p}$. Given a sequence \bc\ we say that the Hamilton cycle $H=(x_1,x_2,\ldots,x_n,x_1)$ (as a sequence of vertices) is \bc-colored if $c(\set{x_i,x_{i+1}})=c_i$ for $i=1,2,\ldots,n$. 

Suppose that $\va=(\a_1,\a_2,\ldots,\a_k)$ where $\a_1,\ldots,\a_k$ are constants and $\a_1+\cdots+\a_k=1$ and $\a_i>0,i=1,2,\ldots,k$. Let $\b=\min\set{\a_i:i\in[k]}^{-1}$ and let $G_{n,p;\va}$ denote the random graph $G_{n,p}$ where each edge is independently given a random color $i$ from the {\em palette} $[k]$ with probability $\a_i$.
\begin{theorem}\label{th1}
Let \bc\ be an arbitrary sequence of colors. Let $p=(\log n+\om)/n$ where $\om\to\infty$. Then w.h.p. $G_{n,\b p;\va}$ contains a \bc-colored Hamilton cycle.
\end{theorem}
\begin{Remark}
In the above theorem we are allowed to take $c_i=\ell,i=1,2,\ldots,n$ for each possible $\ell\in [k]$ and so we cannot improve the $\b p$ probability threshold. This is because each subgraph induced by a single color must itself be Hamiltonian.
\end{Remark}
\begin{Remark}
As will be seen, the proof of Theorem \ref{th1} can be repeated verbatim for the random digraphs $D_{n,,p}$ and $D_{n,\b p;\va}$.
\end{Remark}
\begin{Remark}
The proof can also be extended without difficulty to deal with Hamilton cycles in edge colored hypergraphs. Here we must let $p$ be the threshold probability for a particular type of Hamilton cycle. These thresholds are known fairly precsely for all except loose Hamilton cycles. See Frieze \cite{F1}, Dudek, Frieze, Loh and Speiss \cite{DFLS} for loose Hamilton cycle thresholds and Dudek and Frieze \cite{DF} and Narayanan and  Schacht \cite{NS} for the remaining types.
\end{Remark}
One can also consider problems where the vertices are colored. Here our results are less tight. Suppose now that there are $k$ colors and each $v\in [n]$ is given a color $c(v)\in[k]$. Let $V_i=\set{v:c(v)=i}$ and assume that $|V_i|=\b_in$ for $i\in[k]$ where $\vb=(\b_1,\b_2,\ldots,\b_k)$ and $\b_1+\b_2+\cdots+\b_k=1$ and $\b_i>0,i\in[k]$ so that each set $V_i$ is of linear size. We denote this randomly colored graph by $G_{n,p}^{\vb}$. We can assume w.l.o.g. that vertices $1,2,\ldots,\b_1n$ are given color 1 and vertices $\b_1n+1,\b_1n+2,\ldots,(\b_1+\b_2)n$ are given color 2 etc. Given a sequence \bc\ we now say that the Hamilton cycle $H=(x_1,x_2,\ldots,x_n,x_1)$ (as a sequence of vertices) is \bc-colored if $c(x_i)=c_i$ for $i=1,2,\ldots,n$. 
\begin{theorem}\label{th2}
Let \bc\ be an arbitrary sequence of colors where each color $j$ appears exactly $\a_jn$ times. Let $p=K\log n/n$ where $K=K(\vb)$ is sufficiently large. Then w.h.p. $G_{n,p}^{\vb}$ contains a \bc-colored Hamilton cycle.
\end{theorem}
We can expand our results by coloring the edges as well as the vertices. We prove two results along these lines. First, suppose that $q\geq n$ and we randomly color each edge with one of $q$ colors. A Hamilton cycle is rainbow colored if each edge has a different color. Using a result of Bell, Frieze and Marbach \cite{BFM} and Han and Yuan \cite{HY} we can strengthen Theorem \ref{th2} to
\begin{theorem}\label{th2a}
Let \bc\ be an arbitrary sequence of colors where each color $j$ appears exactly $\b_jn$ times. Let $p=K\log n/n$ where $K=K(\vb)$ is sufficiently large. Suppose in addition that the edges of $G_{n,p}^{\vb}$ are randomly colored with one of $q\geq n$ colors. Then w.h.p. $G_{n,p}^{\vb}$ contains a \bc-colored rainbow Hamilton cycle.
\end{theorem}
We can also consider a combination of Theorems \ref{th1} and \ref{th2}.
\begin{theorem}\label{1+2}
Let $\bc_1=(c_{1,1},c_{1,2},\ldots,c_{1,n})$ be an arbitrary sequence of colors from the palette $[k]$ and let $\bc_2=(c_{2,1},c_{2,2},\ldots,c_{2,n})$ be another arbitrary sequence of colors from the palette $[\ell]$ where each color $j\in [\ell]$ appears exactly $\b_jn$ times. Let $p=K\log n/n$ where $K$ is sufficiently large. Suppose that each edge of $G_{n,p}$ is given a random color from palette $[k]$, using distribution $\va$ and exactly $\b_jn=\Omega(n)$ vertices are given color $j$ for $j\in[\ell]$. Denote this coloring of $G_{n,p}$ by $G_{n,p;\va}^{\vb}$. Then w.h.p. $G_{n,p;\va}^{\vb}$ contains a Hamilton cycle in which the edges follow pattern $\bc_1$ and the vertices follow pattern $\bc_2$.
\end{theorem}
\subsubsection{Prior work on randomly colored random graphs}
\paragraph{Rainbow Hamilton Cycles} 
The most well-studied case is that of rainbow Hamilton cycles. Here we are given $k\geq n$ colors which are applied randomly to the edges of $G_{n,p}$. A rainbow Hamilton cycle is one where each edge has a different color. Cooper and Frieze \cite{CF} showed that if $k\geq 20n$ and $p\geq \frac{20\log n}{n}$ then a randomly colored $G_{n,p}$ contains a rainbow hamilton cycle w.h.p. This was improved to $k\geq n+o(n)$ and $p\sim \frac{\log n}{n}$ by Frieze and Loh. Currently the strongest result is that of Ferber and Krivelevich \cite{FK} who prove a hitting time result when $k=n+o(n)$.
\paragraph{Repeating Patterns} Special cases of Theorem \ref{th1} were proved by Espig, Frieze and Krivelevich \cite{EFK} and by Anastos and Frieze \cite{AF}. Here the sequence \bc\ is required to consist of the repetition of some fixed bounded length subsequence. In this case it was possible to prove hitting time results. Chakraborti, Frieze and Hasabnis \cite{CFH} proved a hitting time version for the existence of patterns where the Hamilton cycle is required to decompose into $k$ concatinated mono-chromatic paths.
\subsection{A fixed order for a subset of vertices} 
Here we consider the following problem. We have a fixed set $S_0\subseteq [n]$ and a fixed ordering of the vertices in $S_0$ and we wish to determine the likelihood that there is a Hamilton cycle that goes through $S_0$ in the given order. We do not require that the vertices of $S_0$ be visited consecutively. Without loss of generality we can assume that $S_0=[s_0]$ and that we wish to find $S_0$ in the natural order.
\begin{theorem}\label{th3}
Let $p=(\log n+\log\log n+\om)/n,\,\om=o(\log\log n)$ and $s_0=\om_1 n/\log n$ where $\om_1=o(\log\log\log n)$. Then w.h.p. $G_{n,p}$ contains a Hamilton cycle in which the vertices $S_0$ appear in natural order.
\end{theorem}
The bound $\om_1=o(\log\log\log n)$ is an artifact of our proof.\\
{\bf Conjecture:} we can replace this bound by $\om_1\leq c_1\log\log n$ for some constant $c_1>0$. 
\subsubsection{Prior work}
The closest result to this is the result of Robinson and Wormald \cite{RW}. They consider random regular graphs and ask for Hamilton cycles that contain a prescribed set of $o(n^{2/5})$ edges that must be contained in order in the cycle.
\subsection{Bounding the number of inversions}
Our final result concerns Hamilton cycles where we place a restiction on the number of invertions in the permutation of $[n]$ that it defines. So we treat a Hamilton cycle $H$ as a sequence $\bi=(i_1=1,i_2,\ldots,i_n)$ and we define $\i(H)=|\set{k<\ell:i_k>i_\ell}|$.
\begin{theorem}\label{th4}
Suppose that $M=\Omega(n\log n)$. There is a constant $K$ such that if $p\geq \frac{Kn\log n}{M}$ then w.h.p. $G_{n,p}$ contains a Hamilton cycle $H$ with $\i(H)\leq M$. Furthermore, if $p\leq p_\e=\frac{(1-\e)n}{eM}$ then w.h.p. $G_{n,p}$ contains no such Hamilton cycle. Here $\e$ is an arbitrary positive constant.
\end{theorem}
We get a restricted rainbow version almost for free:
\begin{theorem}\label{th4a}
Suppose that the edges of $G_{n,p}$ are randomly colored with one of $q\geq n$ colors. There is a constant $K=K(\e)$ such that if $p\geq \frac{K\log n}{n}$ then w.h.p. $G_{n,p}$ contains a rainbow Hamilton cycle $H$ with $\i(H)\leq \e n^2$.
\end{theorem}
In general, except for the case $M=\Omega(n^2)$, there is a $\log n$ gap between the upper and lower bound in Theorem \ref{th4}. (The gap is smaller for $M=\Omega(n^2)/\om,\om=o(\log n)$.) We will be able to remove this gap by studying a greedy algorithm from Frieze and Pegden \cite{FP}.
\begin{theorem}\label{greedy}
If $M\leq Kn^2/\log^2n$ and $p\geq \frac{100\max\set{K,1}n}{M}$ then w.h.p. $G_{n,p}$ contains a Hamilton cycle $H$ with $\i(H)\leq M$. 
\end{theorem}
\section{Proof of Theorem \ref{th1}}\label{sec1}
Let $N=\binom{n}{2}$ and consider the following sequence of (partially) edge colored graphs $\G_m,m=0,1,\ldots,N$. Let $e_1,e_2,\ldots,e_N$ be an enumeration of the edges of $K_n$. To construct $\G_t$ we include $e_1,e_2,\ldots,e_t$ independently with probability $kp$ and give each included edge a random color using distribution $\va$. Then for $i>t$ we include each edge independently with probability $p$. Thus $\G_0$ is a copy of $G_{n,p}$ and $\G_N$ is a copy of $G_{n,kp;\va}$. 

A Hamilton cycle $H=(e_{\p(i)},i=1,2,\ldots,n)$ (as a sequence of edges) of $\G_t$ is $(\bc,t)$-{\em proper} if $c(e_{\p(j)})=c_j$ for $\p(j)\leq t$. Let $\cG_t$ denote the set of graphs containing a $(\bc,t)$-proper Hamilton cycle. 
\begin{lemma}\label{comp}
\[
\Pr(\G_t\in\cG_t)\leq \Pr(\G_{t+1}\in\cG_{t+1})\text{ for }t\geq 0.
\]
\end{lemma}
\begin{proof}
We use a modification of the coupling argument of McDiarmid \cite{McD}. The status of edge $e_i$ consists of (i) whether or not it is included and (ii) its color if $i\leq t$. We condition on the {\em identical} status of the edges $e_i,i\neq t+1$ in $\G_t,\G_{t+1}$ and argue about the conditional probability of both graphs having a (\bc,t)-proper Hamilton cycle. Denote these conditional probabilities by $p_t,p_{t+1}$ respectively. The conditional probability space is now just the status of $e_{t+1}$ in $\G_t,\G_{t+1}$. We argue that $p_t\leq p_{t+1}$. Let $\wh\G$ denote the subgraph induced by the edges $e_i,i\neq t+1$ whose status means they are included in $\G_t$ and $\G_{t+1}$. (Thus $\wh\G$ is only partially edge colored.) There are several cases: 
\begin{enumerate}
\item $\wh\G\in \cG_t\cap \cG_{t+1}$. In this case $p_t=p_{t+1}=1$.
\item $\wh\G+e_{t+1}\notin\cG_t\cup \cG_{t+1}$, regardless of the status of $e_{t+1}$. In this case $p_t=p_{t+1}=0$.
\item Failing 1. and 2. we consider the case where $\wh\G$ is such that the existence of the edge $e_{t+1}$ matters. We consider the event (i) that including $e_{t+1}$ creates a $(\bc,t)$-proper Hamilton cycle $P+e_{t+1}$ in $\G_t$ and the event (ii) that including $e_{t+1}$ with an appropriate color creates a $(\bc,t)$-proper Hamilton cycle $P+e_{t+1}$ in $\G_{t+1}$. In this case we see that 
\[
p_{t+1}=\Pr((ii))\geq \min\set{\b p\a_i:i\in[k]}\geq p=\Pr((i))=p_{t}.
\]
\end{enumerate}
\end{proof}
This proves Theorem \ref{th1}.
\section{Proof of Theorems \ref{th2} and  \ref{th2a}}\label{th2th2a}
For this theorem we will use the breakthrough result of Frankston, Kahn, Narayanan and Park \cite{FKNP}. Recall the setup in \cite{FKNP}: A hypergraph $\cH$ (thought of as a set of edges) is $r$-bounded if $e\in \cH$ implies that $|e|\leq r$.  For a set $S\subseteq X=V(\cH)$ we let $\upp{S}=\set{T:\;S\subseteq T\subseteq X}$ denote the subsets of $X$ that contain $S$. Let $\upp{\cH}=\bigcup_{H\in \cH}\upp{H}$ be the collection of subsets of $X$ that contain an edge of $\cH$. We say that $\cH$ is $\k$-spread if we have the following bound on the number of edges of $\cH$ that contain a particular set $S$: 
\beq{spread}{
|\cH\cap \upp{S}|\leq \frac{|\cH|}{\k^{|S|}},\quad\forall S\subseteq X.
}
Let $X_p$ denote a subset of $X$ where each $x\in X$ is included independently in $X_p$ with probability $p$. The following theorem is from \cite{FKNP}:
\begin{theorem}\label{th3a}
Let $\cH$ be an $r$-bounded, $\k$-spread hypergraph and let $X=V(\cH)$. There is an absolute constant $C>0$ such that if
\beq{mbound}{
p\geq\frac{C\log r}{\k}
}
\end{theorem}
then w.h.p.~$X_p$ contains an edge of $\cH$. Here w.h.p.~assumes that $r\to\infty$.

Bell, Frieze and Marbach \cite{BFM} and He and Yuan \cite{HY} proved a rainbow version of Theorem \ref{th3a}. 
\begin{theorem}\label{thrainbow}
Let $\cH$ be an $r$-bounded, $\k$-spread hypergraph and let $X=V(\cH)$ be randomly colored from $Q=[q]$ where $q\ge r$. Suppose also that $\k=\Omega(r)$. There is an absolute constant $C>0$ such that if 
\beq{mbounds}{
p\geq\frac{C\log r}{\k}
}
then w.h.p.~$X_p$ contains a rainbow colored edge of $\cH$. Here w.h.p.~assumes that $r\to\infty$ (and thus $\k\to\infty$).
\end{theorem}
(In truth the general theorem in \cite{BFM} only proves that for a given $\e>0$ there is a constant $C_\e>0$ such if $C\geq C_\e$ then $X_p$ contains a rainbow colored edge of $\cH$ with probability $1-\e$. We have to apply a Theorem of Friedgut \cite{F05} (second remark following Theorem 2.1 of that paper) to obtain w.h.p. The paper \cite{HY} has subsequently removed these restrictions.)

In our use of Theorem \ref{thrainbow} we let $X=\binom{[n]}{2}$. Each $x=\set{u,v}\in X$ will have colored endpoints $\set{c(u),c(v)}$. Our hypergraph $\cH$ consists of sets of $n$ edges with colored endpoints that together make up a \bc-colored Hamilton cycle. We now check that \eqref{mbound} holds with $\k=\Omega(n)$.

We first observe that $|\cH|=\frac{1}{h}\prod_{i=1}^kn_i!$ where $n_i=|V_i|$ and $h$ is the number of automorphisms of a \bc-colored Hamilton cycle. Fix a set $S$ for which $\f(S):=|\cH\cap \upp{S}|>0$. In particular, $S$ is the edge-set of a collection of paths.  We bound $\phi(S)$ as follows.  If $S$ consists of the edges of $t$ paths, then to choose a Hamilton cycle consistent with $S$, we must
\begin{itemize}
\item Choose an orientation of each of the $t$ paths;
\item Choose, for the starting vertex of each of the $t$ paths, where its index is in $1,\dots,n$ in the Hamilton cycle, whose color must match the color of the starting vertex;
\item Choose, for each vertex not incident with any edge in $S$, its index in $1,\dots,n$ in the Hamilton cycle, whose color must match the color of the starting vertex.
\end{itemize}
Of course many such choices will not give rise to valid \bc-colored Hamilton cycles, but any valid \bc-colored Hamilton cycles consistent with $S$ can be specified by such choices.

Note that, as above, after choosing the orientation of the paths, we have to choose the index of at most $n-s$ vertices where $s=|S|$, since the number of components (paths or isolated vertices) of the graph induced by the set $S$ is $n-s$, and we only choose the index of the first vertex in each path.  In particular, after choosing the orientation of the path, suppose that we need to choose the index of $n_j-s_j$ vertices of color $j$ for each $j$, where $\sum s_j=s$.  Now, as $n_i=\a_in,j=1,\ldots,k$, we have the following bound on the number of choices:
\[
\frac{1}{h}\prod_{j=1}^k(n_j-s_j)!=|\cH|\prod_{j=1}^k\frac{(n_j-s_j)!}{n_j!}\leq |\cH|\prod_{j=1}^k\frac{e^{s_j^2/n_j}}{n_j^{s_j}}\leq e^s|\cH|\prod_{j=1}^k\frac{1}{n_j^{s_j}}\leq e^s|\cH|\frac{1}{n^s\a_{\min}^s}.
\]

This gives that 
\[
\phi(S)\leq 2^s|\cH|\frac{e^s}{n^s\a_{\min}^s}
\]
So, \eqref{mbound} holds with $r=n$ and $\k=\a_{\min}n/2e$. This proves Theorem \ref{th2a} (which implies Theorem \ref{th2}, perhaps with a smaller hidden constant $C$).
\section{Proof of Theorem \ref{1+2}}
For the proof of this theorem we combine McDiarmid's coupling with Theorem \ref{th2}. We use the notation of Section \ref{sec1}. We define the sequence $\G_1,\G_2,\ldots,\G_N$ similarly to how we did in that section but with the difference that we have colored the vertices as claimed.  

A Hamilton cycle $H=(e_{\p(i)},i=1,2,\ldots,n)$ (as a sequence of edges) of $\G_t$ and equal to $(v_{1},i=1,2,\ldots,n)$ (as a sequence of vertices) is $(\bc_1,\bc_2,t)$-{\em proper} if $c(e_{\p(j)})=c_{1,j}$ for $\p(j)\leq t$ and $c(v_i)=c_{2,i}$ for $i=1,2,\ldots,n$. Let $\cG_t$ denote the set of graphs containing a $(\bc_1,\bc_2,t)$-proper Hamilton cycle. 
\begin{lemma}
\[
p_t=\Pr(\G_t\in\cG_t)\leq p_{t+1}=\Pr(\G_{t+1}\in\cG_{t+1})\text{ for }t\geq 0.
\]
\end{lemma}
\begin{proof}
The proof of this is identical to that of the proof of Lemma \ref{comp} except that we replace $(\bc,t)$-proper by $(\bc_1,\bc_2,t)$-proper.
\end{proof}
Now $\G_0$ is distributed as $G_{n,p}^{\vb}$ and Theorem \ref{th2} states that a $(\bc_1,\bc_2,0)$-proper Hamilton cycle exists w.h.p. On the other hand, $\G_N$ is distributed as $G_{n,p;\va}^{\vb}$ and the lemma implies that w.h.p. it contains a $(\bc_1,\bc_2,N)$-proper Hamilton cycle which is what we need to prove. This proves Theorem \ref{1+2}.

\section{Proof of Theorem \ref{th3}}
We begin by generating $G=\bigcup_{i=1}^4\G_i$ where each $\G_i,i=1,2,3,4$ is an independent copy of $G_{n,p_i}$. Here $p_3=p_4=\om/4n$ and $p_1=p_2$ and  $1-p=\prod_{i=1}^4(1-p_i)$. It follows that $p_1=p_2\sim p/2$. For $v\in[n]$, we let $d(v)$ denote the degree of vertex $v$ in $G_{n,p}$ and $d_i(v),i=1,2$ denote the degree of $v$ in $\G_i$.

\newcommand{\SMALL}{\mathrm{SMALL}}
\newcommand{\TINY}{\mathrm{TINY}}
\newcommand{\AVOID}{\mathrm{AVOID}}
\newcommand{\END}{\mathrm{END}}

Let $G_1=\G_1\cup\G_2$. $G_1$ has minimum degree at least 2, w.h.p.   We let $\SMALL=\set{v:d_{G_1}(v)\leq \frac{1}{40}\log n}$. An easy first moment calculation implies that w.h.p. 
\begin{enumerate}[S(i)]
\item $|\SMALL|\leq n^{1/3}$.
\item $v,w\in \SMALL$ implies that $dist(v,w)\geq 5$.\label{vwseparate}
\item No cycle of length less than 5 contains a vertex of $\SMALL$.
\end{enumerate}
The calculations supporting this claim can be found in Lemma 3.1 of \cite{BFF}. We note next that w.h.p. $G_{n,p}$ has maximum degree at most $5\log n$. This follows from a simple first moment calculation, again given in Lemma 3.1 of \cite{BFF}. 

Similarly, we let $\TINY_0=\set{v:d_1(v)\leq \frac{1}{40}\log n}$ and initialise $\TINY=\TINY_0$.   The graph $\G_1$ shrinks as our construction progresses in that we remove vertices from $\G_1$ and place them in $\TINY$ when their $\Gamma_1$-degrees become smaller than $\tfrac 1 {80}\log n$. We will show that $\TINY_0$ is small in Lemma \ref{tiny} below. Initialise the set $\AVOID=\SMALL\cup N(\SMALL)\cup \TINY$. (Here $N(S)=\set{w\notin S:\exists v\in S\text{ such that \{v,w\} is an edge of $G$}}$.)

The construction of our Hamilton cycle goes as follows. 
\begin{enumerate}[Step 1]
\item For $v=1,2,\ldots,s_0$ we construct a set of vertex disjoint paths $P_v=(x_v,\ldots,v,\ldots,y_v)$ where $x_v,y_v\notin \TINY$.  (If $v\notin \TINY$ then we can simply let $x_v=y_v=v$ and $P_v=v$).  These paths are of length at most 6. These paths will avoid using vertices in $\AVOID\cup (S_0\setminus\set{v})\cup\bigcup_{w<v}V(P_w)$. Also, for $v\geq 2$, we avoid using vertices in $N(x_1)$. All edges except perhaps those incident with $x_v,y_v$ will be from $\G_2$. After we create a path, we delete the vertices in the interior of the path and their incident edges.
\item We then use the edges of $\G_1$ to construct vertex disjoint paths $Q_v$ from $y_v$ to $x_{v+1}$ for $v=1,2,\ldots,s_0-1$.  They will be of length at most $4\log n/\log\log\log n$. These paths will avoid using vertices in $\AVOID\cup S_0$. After we create a path, we delete the vertices in the interior of the path and their incident edges. If after this deletion the $\G_1$-degree of a vertex becomes at most $\log n/80$ then we add it to $\TINY$ and update $\AVOID$.
\item We let $P^*=(P_1,Q_1,P_2,Q_2,\ldots,Q_{s_0-1},P_{s_0})$ and let  $x^*,y^*\notin\TINY$ be its endpoints. Here $x^*=x_1$ and $y^*=y_{s_0}$. 
\item We then use the extension-rotation algorithm to find a Hamilton cycle  that contains $P^*$ as a subpath.
\end{enumerate}
We note that because of our bound on $s_0$,
\beq{pathsize}{
\sum_{v=1}^{s_0}(|V(P_v)|+|V(Q_v)|)\leq s_0\brac{7+\frac{4\log n}{\log\log\log n}}=o(n).
}
\subsection{Analysis of Step 1}
We first show that $|\TINY|_0$ is small.
\begin{lemma}\label{tiny}
$|\TINY_0|\leq n^{2/3}$ w.h.p.
\end{lemma}
\begin{proof}
We have
\mults{
\E(|\TINY_0|)\leq n\sum_{k=0}^{\log n/40}\binom{n}{k}p_1^k(1- p_1)^{n-1-k}\leq n\sum_{k=0}^{\log n/40}\bfrac{e^{1+o(1)}\log n}{2k}^kn^{-1/2}\\
\leq 2n(20e^{1+o(1)})^{\log n/40}n^{-1/2}\leq n^{3/5}.
}
The lemma now follows from the Markov inequality.
\end{proof}
Fix $v\in S_0$. Expose the edges of $G_1$ incident with $v$ and let $A_1$ denote the other endpoints of these edges. Assume first that $v\in \TINY\setminus\SMALL$, so that $|A_1|\geq \tfrac 1 {40}\log n$. We go through the vertices of $A_1$ in order until we find a vertex $a_v$ with a $\G_2$-neighbor $x_v\notin \AVOID\cup N(x_1) \cup\bigcup_{j=1}^{v-1}V(P_j)$. Lemma  \ref{tiny} and Lemma \ref{V1G} below ensures that $|\TINY|\leq 3n/5$, which implies that
\[| \AVOID\cup N(x_1) \cup\bigcup_{j=1}^{v-1}V(P_j)|\leq 3n/5+o(n).\]
If $\eta_v$ denotes the number of trials to find $a_v$ then $\eta_v$ is dominated by a geometric random variable with success probabilty at least $\brac{1-(1-p_2)^{2n/5+o(n)}}$ and so $\Pr(\eta_v\geq 10)=o(n^{-2})$. This verifies the existence of $a_v,x_v$. The rest of $P_v$ is justified similarly, we just find another path, avoiding $a_v,b_v,x_v$ as well. If $v\in \TINY\cap\SMALL$ then we choose two arbitrary neighbors of $v$, which will not be in $\SMALL$ by S(\ref{vwseparate}) and grow paths to $[n]\setminus \TINY$, avoiding $\AVOID\cup N(x_1)\cup \bigcup_{j=1}^{v-1}V(P_j)$.
\subsection{Analysis of Step 2}
We first remove vertices in $\AVOID$ from $\G_1$. We constuct $Q_1,Q_2,\ldots,Q_{s_0}$ in this order and at each step $v\geq 1$, we do the following in the graph $\G_1$: we remove the vertices of $Q_{v-1}$ (if $v\geq 2$) and then repeatedly remove vertices of degree (in $\G_1$) at most $\frac{1}{80}\log n$ until what remains has minimum degree at least $\frac{1}{80}\log n$. (Removed vertices are placed into $\TINY$.) We show that w.h.p. 
\begin{enumerate}[Property 1]
\item $|V(\G_1)|\geq 2n/5$ throughout.
\item The diameter of $\G_1$ is at most $4\log n/\log\log\log n$ throughout.
\end{enumerate}
We will in fact halt the construction and declare failure if either property fails to hold. Suppose we are constructing a path $Q_v$. What we argue is that as long as we have Property 1, we will w.h.p. have good expansion in $\G_1$. This will ensure that $Q_v$ is short, which will then be used to show that Property 1 continues to hold after we delete $Q_v$. So, there is no circular argument.

We begin with the following lemma:
\begin{lemma}
Suppose that $S\subseteq[n]$. In the graph $G$,
\begin{align}
&|S|\geq \s_0=\frac{5n}{\log n\log\log n}\text{ implies that }e(S)\leq \frac{3|S|^2\log n(\log\log n)^2}{2n}.\label{eq1}\\
&S|\geq \s_1=\frac{n}{4000(\log\log n)^2}\text{ implies that }e(S)\leq \frac{2|S|^2\log n}{n}.\label{eq1a}\\
&|S|\leq \s_0\text{ implies that }e(S)\leq \frac{7|S|\log n}{\log\log n}.\label{eq2}\\
&|S|=\s_0\text{ implies that }|N_1(S)|\geq \frac{n}{(\log\log n)^2},\label{eq3}
\end{align}
where $N_1(S)$ is the set of vertices not in $S$ that have a $\G_1$-neighbor in $S$.
\end{lemma}
\begin{proof}
Let $s=|S|$. Then, 
\begin{align}
\Pr(\exists S:e(S)\geq \a s)&\leq\binom{n}{s}\binom{s^2/2}{\a s}p^{\a s}\nn\\
&\leq \brac{\frac{ne}{s}\cdot\bfrac{se^{1+o(1)}\log n}{2\a n}^{\a}}^s.\label{alp}
\end{align}
For \eqref{eq1} we use $\a=3s\log n(\log\log n)^2/2n\geq 15\log\log n/2$, for \eqref{eq1a} we use $\a=2s\log n/n$ and for \eqref{eq2} we use $\a=7\log n/\log\log n$. In all cases the R.H.S. of \eqref{alp} is $o(1)$.

The edges in $\G_1$ are conditioned so that the minimum degree is at least $\log n/80$. Without the conditioning, if $|S|=\s_0$ then $|N_1(S)|$ is distributed as $Bin(|V(\G_1)|-\s_0,1-(1-p)^{\s_0})$. Now $1-(1-p)^{\s_0}\geq 1-e^{-5/\log\log n}\geq 4/\log\log n$. So, by Chernoff bounds and the FKG inequality and by Property 1,
\mults{
\Pr\brac{\exists S:|S|=\s_0,|N_1(S)|\leq \frac{n}{(\log\log n)^2}}\leq \binom{n}{\s_0}  \exp\set{-\frac{2n/5-\s_0}{2+o(1)}\cdot\frac{4}{\log\log n}}=\\
\exp\set{O(\s_0\log\log n)-\Omega\bfrac{n}{\log\log n}}=o(1).
}
\end{proof}
We now argue that Property 1 holds throughout.
\begin{lemma}\label{V1G}
W.h.p. we have $|V(\G_1)|\geq 2n/5$ thoughout.
\end{lemma}
\begin{proof}
Suppose we delete $\g n$ vertices belonging to paths and then repeatedly remove vertices of degree at most $\frac{1}{80}\log n$ from $\G_1$. Here $\g n$ bounds the the total length of all the paths $Q_1,Q_2,\ldots,Q_{s_0}$ and is $o(n)$, see \eqref{pathsize}.  Initially, $\G_1$ has at least $(1/2-o(1))n\log n/2$ edges and after $\k n$ small-degree-vertex removals $\G_1$ has $(1-\g-\k)n$ vertices and at least $(1/4-o(1)-5\g-\k/80)n\log n$ edges. (We lose at most $5\log n$ edges per path vertex and at most $\log n/80$ edges per low degree vertex.) It follows from \eqref{eq1a} that w.h.p. 
\[
\brac{\frac14-o(1)-\frac{\k}{80}}\leq \frac{3(1-\g-\k)^2}{2}\text{ implying that }\brac{\frac14-o(1)-\frac{1}{80}}\leq \frac{3(1-\g-\k)^2}{2}.
\]
It follows from this that w.h.p. $\G_1$ still has at least $(1-\g-\k)n\geq 2n/5$ vertices.
\end{proof}
This verifies Property 1.
\begin{lemma}\label{V2G}
W.h.p., the diameter of $\G_1$ is at most $4\log n/\log\log\log n$ thoughout.
\end{lemma}
\begin{proof}
Fix $x,y\in V(\G_1)$ and let $S_i$ denote the set of vertices at distance $i$ from $x$ in $\G_1$ and define $T_i$ similarly for $y$. Fix $i$ and let $S=S_i\cup S_{i+1}$. Suppose that $|S|\leq \frac{5n}{\log n\log\log n}$. The minimum degree in $S$ is at least $\log n/80$ and given \eqref{eq2}, we have
\[
\frac{|S_i|\log n}{160}\leq e(S)\leq \frac{7|S|\log n}{\log\log n}.
\]
It follows that $|S_{i+1}|\geq \frac{1}{2000}|S_i|\log\log n$. Let $i_0\leq 3\log n/2\log\log\log n$ be the smallest positive integer such that $(\frac{1}{2000}\log\log n)^i\geq 5n/\log n\log\log n$. If $S_{i_0}\cap T_{i_0}\neq\emptyset$ then there is a path of length at most $4\log n/\log\log\log n$ from $x$ to $y$ in $\G_1$. Now choose sets $S\subseteq S_{i_0},T\subseteq T_{i_0}$ of size exactly $\s_0$. On the other hand, $|N_1(S)|,|N_1(T)|\geq n/(\log\log n)^2$ and in $G_1$,
\begin{align*}
&\Pr\brac{\exists S,T;|S|,|T|\geq \frac{n}{(\log\log n)^2},e(S:T)=S\cap T=\emptyset}\\
&\leq \binom{n}{n/(\log\log n)^2}^2(1-p_1)^{(n/(\log\log n)^2)^2}\\
&\leq \brac{e(\log\log n)^2}^{2n/(\log\log n)^2}\exp\set{-\frac{ n\log n}{2(\log\log n)^4}}=o(1).
\end{align*}
It follows that w.h.p. the diameter of $\G_1$ is $4\log n/\log\log\log n$. 
\end{proof}
This verifies Property 2.

Now we argue that we can construct the paths $Q_i$.  Given Properties 1 and 2, we use the edges of $\G_2$ to find a short paths from $x_i$ and $y_i$ to what's left of $V(\G_1)$ at the time of the construction of $Q_i$. The path $Q_i$ will be decomposed into $Q_{i,j}, j=1,2,3$ where $Q_{i,1}$ uses $\G_2$-edges and is from $x_i$ to $u_i\in V(\G_1)$, $Q_{i,2}$ is from $u_i$ to $v_i\in V(\G_1)$ in $\G_1$ and $Q_{i,3}$ uses $\G_2$-edges and is from $v_i$ to $y_i$. We let $A_j$ denote the set of vertices at distance exactly $j$ from $x_i$ using paths that avoid using $\bigcup_{j=1}^{i-1}V(Q_j)$. Because $x_i\notin \TINY$, we have $|A_1|\geq \log n/40$ and  then given $A_j$, $|A_{j+1}|$ dominates the binomial $Bin(|V(\G_1)|-|A_1|-\cdots-|A_{j-1}|-o(n),1-(1-p_2)^{|A_j|})$. So w.h.p. $|A_2|=\Omega(\log^2n)$ and $A_3\cap V(\G_1)\neq \emptyset$, using Lemma \ref{V1G}. This verifies the existence of $Q_{i,1}$ and $Q_{i,3}$ is dealt with similarly. 

In summary Step 2 constructs a path $P^*$ of length $O(s_0\log n/\log\log\log n)=o(n)$.
\subsection{Analysis of Step 3}
We only need to verify that $x^*=x_1,y^*=y_{s_0}\notin \TINY$. $x_1^*\notin \TINY$ because it was not in $\TINY_0$ and we avoid using vertices in $N(x_1)$. $y_{s_0}$ is selected to be not in $\TINY$ at the end of the process.
\subsection{Analysis of Step 4}
Let $G_1^*$ be obtained from $G_1$ after contracting $P^*$ to an edge $e^*=\set{x^*,y^*}$ and deleting any edge $\set{u,v}$ that is incident with an interior vertex of $P^*$, but is not an edge of $P^*$. We let $V_1^*=V(G_1^*)$. We note that the minimum degree in $G_1^*$ is at least 2. We then let $G_3^*=G_1^*\cup \G_3^*$, where $\G_j^*,j=3,4$ is the subgraph of $G_j$ induced by $V_1^*$. 

We consider the usual extension-rotation algorithm for finding a Hamilton cycle. We apply it to $G_3^*$ and we use $\G_4^*$ as {\em boosters}. We begin with a longest path in $P_0=(x_1,x_2,\ldots,x_k)$ in $G_3^*$ that contains $e^*$ and we consider {\em restricted} rotations that do not delete $e^*$. Given a path $P=(x_1,x_2,\ldots,x_k)$ and an edge $\set{x_k,x_i}$ where $1<i<k-1$ we say that the path $Q=(x_1,\ldots,x_{i-1},x_i,x_k,x_{k-1},\ldots,x_{i+1})$ is obtained from $P$ by a restricted rotation if $e^*\neq \{x_i,x_{i+1}\}$. $x_1$ is called the fixed endpoint.

Suppose then that $\END$ is the set of vertices that occur as endpoints of paths obtainable from $P_0$ by a sequence of restricted rotations. Since $P_0$ was a longest path containing $e^*$, the neighbors of $\END$ in $G_3^*$ are all vertices of $P_0$. We show that we have something close to the usual Pos\'a property. For $S\subseteq V_1^*$ we let $N^*(S)$ denote the neighbors of $S$ in $G_1^*$. 
\begin{lemma}\label{posa}
$|N^*(\END)|\leq 2|\END|+1$.
\end{lemma}
\begin{proof}
We show that
\beq{end}{ 
|N^*(\END)\setminus \set{x^*,y^*}|\leq 2|\END|-1,
}
This implies the lemma. If $u\in \END$ and $v\in P_0\setminus (\END\cup\set{x^*,y^*})$ is a neighbor of $u$ in $G_1$ then one of $v$'s neighbors in $P_0$ must be in $\END$. This is because when the rotations produce a path $P_u$ with $u$ as an endpoint, another rotation will make one of $v$'s $P_v$ neighbors an endpoint. If both of these neighbors are $P_0$ neighbors then we are done. Otherwise the rotations have deleted a $P_0$-edge $\set{x,v}\neq e^*$ containing $v$. This means $x$ or $v$ is in $\END$. Our hypothesis excludes $v\in \END$. This completes the proof of \eqref{end}.
\end{proof}
To apply the usual arguments, we prove
\begin{lemma}\label{conn}
The following hold w.h.p.: 
\begin{enumerate}[(a)]
\item $S\subseteq V_1^*,|S|\leq n/6000$ implies that $|N^*(S)|\geq\sum_{v\in S_1}d(v)+2|S_2|$, where $S_1=S\cap \SMALL$ and $S_2=S\setminus S_1$.
\item $G_3^*$ is connected.
\end{enumerate}
\end{lemma}
\begin{proof}
(a) Suppose first that $S\cap \SMALL=\emptyset$ and $|N^*(S)|\leq 5|S|$. Let $T=S\cup N^*(S)$. Vertices in $S$ have degree at least $\log n/40$ in $G_3^*$ and so $e(T)\geq |S|\log n/80\geq |T|\log n/480$. If $|T|\leq \frac{5n}{\log n\log\log n}$ then this contradicts \eqref{eq2}. Otherwise, \eqref{eq1} is contradicted, unless $3|T|^2\log n(\log\log n)^2/2n\geq |T|\log n/480$ which implies that $|S|\geq n/4000(\log\log n)^2$. But if $|T|\geq n/4000(\log\log n)^2$ then \eqref{eq1a} implies that $|T|\geq n/960$ and then $|S|\geq n/6000$.

Suppose now that $S\subseteq V_1^*$ with $|S|\leq n/6000$. Then from the properties S(ii), S(iii) claimed at the beginning of the proof of Theorem \ref{th3}, we have 
\begin{align}
|N^*(S)|&\geq |N^*(S_1)|+|N^*(S_2)|-|N^*(S_1)\cap S_2|-|N^*(S_2)\cap S_1|-|N^*(S_1)\cap N^*(S_2)|\nn\\
&\geq \sum_{v\in S_1}d(v)+5|S_2|-|S_2|-|S_2|-|S_2|\nn\\
&=\sum_{v\in S_1}d(v)+2|S_2|.\label{low}
\end{align}
(b) We first claim that w.h.p., the graph $\G_3^*$ consists of a giant component plus $o(n)$ small components of size at most $n_0=ne^{-\om/10}$.  To verify the claim, let $X_k$ denote the number of components in $\G_3^*$ of size $k\in[n_0,m]$, where $m=|V_1^*|$. Then
\[
\E\brac{\sum_{k=n_0}^{m/2}X_k}\leq \sum_{k=n_0}^{m/2}\binom{m}{k}k^{k-2}\bfrac{\om}{4n}^{k-1}\brac{1-\frac{\om}{4n}}^{k(m-k)}\leq m\sum_{k=n_0}^{m/2}\bfrac{e^{1-\om/8}k\om}{n_0}^k\leq n^2(e^{-\om/50})^{n_0}=o(1).
\]
So, w.h.p. there are no components of size in the range $[n_0,m/2]$. We also have
\[
\E\brac{\sum_{k=1}^{n_0}kX_k}\leq m\sum_{k=1}^{n_0}k\bfrac{e^{1-\om/8}\om}{k}^k=O(ne^{-\om/10}).
\]
The Markov inequality implies that there are $o(n)$ vertices in components of size at most $n_0$. So w.h.p. $\G_3^*$ has a unique giant component of size $m-o(n)$.

Part (a) shows that the minimum component size in $G_3^*$ is at least $n/6000$. This combined with the fact that $\G_3^*$ consists of a giant component of size $m-o(n)$ proves Part (b).
\end{proof}
Now because $d_{G_3^*}(v)\geq 2$ for $v\in[n]$, we can see from Lemma \ref{posa} and \eqref{low} that $|\END|>n/6000$ unless $\END\subseteq \SMALL$ and at most one vertex of $\END$ has degree more than 2. We rule out this possibility. Suppose that $v\in \END\cap \SMALL$. Then $v$ is not adjacent to $x^*$ or $y^*$, by construction. One more rotation will bring a vertex of $[n]\setminus \SMALL$ into $\END$, contradiction.

It follows from Lemmas \ref{posa} and \ref{conn}(b) that w.h.p. $|\END|\geq cn,c=1/6000$. For each $v\in \END$, we can define a set $\END(v)$ of at least $cn$ vertices obtainable by doing rotations with $v$ as the fixed endpoint. 

We can now use a standard argument, see for example Chapter 6.2 of \cite{FrKa}, to use $\G_4^*$ to create the required Hamilton cycle. It will be convenenient to replace the edges of $\G_4^*$ by $\m=\om n/10$ random edges $\set{f_1,f_2,\ldots,f_\m}$. These edges are independent of $G_3^*$. Starting with $i=0$ we construct a sequence of paths $P_0,P_1,\ldots,P_s$ where $s$ is a Hamilton path and construct a Hamilton cycle from there. Given $P_i$ we do restricted rotations until either (i) we construct a path $P$, one of whose endpoints has a neighbor outside $P$ or (ii) we construct at least $cn$ sets $\END(v)$, each of size at least $cn$. In the former case (i) we just extend $P$ to a path $P_{i+1}$ which has one more vertex than $P$. In the latter case (ii) we go to the next edge $f_j$ in the sequence $f_1,f_2,\ldots,f_\m$ to see if it is of the form $\set{x,y},x\in \END(y)$. This closes a path to a cycle. The probability of this is at least $c^2$. Given such a cycle $C$ and the fact that $G_3^*$ is connected, there are two possibilities: (a) $C$ is a Hamilton cycle or (b) there is an edge $\set{x,y}$ such that $x\in V(C)$ and $y\notin V(C)$. We can delete an edge $e\neq e^*$ of $C$ such that we obtain a new path with endpoint $y$ that is one edge longer than $P_i$. This will be our $P_{i+1}$. The probability this process fails is at most the probability that $\m$ trials with success probability $c^2$ fails to produce $n$ successes, which is $e^{-\Omega(n)}$.

This completes the proof of Theorem \ref{th3}.
\section{Proof of Theorem \ref{th4}}
We will use Theorem \ref{th3a} to prove this. We must first prove bounds on the number of Hamilton cycles $H$ with a bound on $\i(H)$. Denote this upper bound on $\i(H)$ by $M$. For a sequence $\bi=(\s_1,\s_2,\ldots,\s_n)$ we let $\m_k=|\set{j<k:\s_j>\s_k}|$ for $k=1,2,\ldots,n$. Then $0\leq \mu_j<j$ for each $j$, and $\i(\bi)=\sum_{j=1}^n\m_j$.  In particular, we have that
\beq{trans}{
|\set{\bi:\i(\bi)\leq M}|=\card{\set{(\m_1,\ldots,\m_{n-1}):\sum_{ j=0}^{n-1}\m_j\leq M,0\leq\m_j<j}}.
}
Indeed, there is a bijection between the sets on the left and the right, realized by building a permutation iteratively in the order $1,2,\ldots,n$ and placing $k$ so that it occurs in front of $\m_k$ previously allocated elements. 

The number of solutions to $\sum_{ j=0}^{n-1}\m_j\leq M$ is bounded by $\binom{M+n}{n}$ and we get our lower bound on $p$ by a first moment calculation. Thus, if $p\leq p_\e$ then since $M=\Omega(n\log n)$, by assumption,
\[
\Pr(\exists H:\i(H)\leq M)\leq \binom{M+n}{n}p_\e^n\leq \brac{\frac{e(M+n)}{n}\frac{(1-\e)n}{e M}}^n=o(1).
\]
We are seeking an upper bound on the threshold probability for the existence of a particular type of Hamilton cycle and so it is acceptable to restrict our attention to a more restrictive subclass of Hamilton cycles. So we restrict our attention to those cycles for which
\beq{bm}{
\sum_{j=1}^{n}\m_j\leq M\text{ and }0\leq \m_j< \begin{cases} j&j\leq  M/n.\\\frac{M}{n}&j> M/n.\end{cases}
}
To apply Theorem \ref{th3a} we let $\cH$ denote the set of Hamilton cycles $H$ such that \eqref{bm} holds. Note that the constraint $\sum_{j=1}^{n}\m_j\leq M$ is redundant in \eqref{bm}. Thus,
\beq{sumL}{
|\cH|=\bfrac M n !\bfrac{M}{n}^{n-M/n}.
}
We prove below that for $S\subseteq X,\,X=\binom{[n]}{2},|S|=s$,
\beq{sizeS}{
|\upp{S}|\leq \bfrac{M}{n}!\bfrac{M}{n}^{n-M/n-s}.
}
It follows from \eqref{sumL} and \eqref{sizeS} that
\[
\frac{|\cH|}{|\upp{S}|}\geq  \bfrac{M}{en}^s.
\]
The upper bound on the existence threshold in Theorem \ref{th4} now follows from Theorem \ref{th3a} with $r=\binom{n}2$ and $\k=M/en$. To obtain the upper bound in Theorem \ref{th4a} we apply Theorem \ref{thrainbow} in place of Theorem \ref{th3a}.
\paragraph{Proof of \eqref{sizeS}:}  As in the proof of Theorem \ref{th2th2a} the set $S$ defines a collection of vertex disjoint sub-paths $P_1,P_2,\ldots,P_k$ of any Hamilton cycle that contains $S$. Given such a path $P_i$ we let $x_i$ denote the lowest numbered vertex of $P_i$. We see that once we have chosen $\m_{x_1}$, the remaining values $\m_{x_i},i\geq 2$ are {\em constrained} by the edges of the cycle $H$ that are not on $P_i$. Let $V_0$ denote the set of first vertices of $P_1,P_2,\ldots,P_k$ and let $V_1=\bigcup_{i=1}^kV(P_i)\setminus V_0$ and $V_2=V_0\cap [M/n]$. Then,
\beq{up1}{
|\upp{S}|\leq \bfrac{M}{n}^{n-M/n-|V_1|+|V_2|}\brac{\frac{M}{n}-|V_2|}!.\\
}
The second factor in \eqref{up1} follows from the additional fact that given the $\m$-values of the elements of $V_2$ there will be $|V_2|$ values forbidden as a $\mu$-value for the unconstrained elements of $[M/n]$. These forbidden values are those that would insert the element into the interior of a path. 

Now $|V_1|=s$ and Stirling's formula implies that $\brac{\frac{M}{n}-|V_2|}!\leq e^{|V_2|}\bfrac{n}{M}^{|V_2|}$. Plugging these into \eqref{up1} yields \eqref{sizeS}. 
\section{Proof of Theorem \ref{greedy}}
We first write $G_{n,p}=G_1\cup G_2$ where the $G_i$ are independent copies of $G_{n,p_i}$, where $p_1=p/3$ and $1-p=(1-p_1)(1-p_2)$. Note that $p_2\sim 2p/3$.   We begin by constructing a path $P_0$ via the following algorithm: We start with $v_1=1$. Then for $j\geq 1$ we let 
$$\f(j)=\min\set{k\in N:k\notin\set{v_1,v_2,\ldots,v_j}\text{ and }\set{v_j,k}\in E(G_1}$$
and let $v_{j+1}=\f(j)$ i.e. we move from $v_j$ to the lowest index $k$ that has not been previously added to $P_0$. We define $U_j$ by  
\[
U_j=\set{i\leq n:i\notin \set{v_1,v_2,\ldots,v_{j}}}
\]
We stop the process at $j=j_0$ when either $|U_j|=\frac{2\log n}{p_1}$ or $v_{j+1}$ does not exist. We then extend the path $P_0=(v_1,v_2,\ldots,v_{j_0})$ to a Hamilton cycle $H$ using the edges of $G_2$ to create a path through $U=U_{j_0}$. 

Observe first that if $|U|=k>\frac{2\log n}{p_1}$ then $\Pr(j_0\leq n-k)\leq n(1-p_1)^k$. This is because at $j_0$ we find that $v_{j_0}$ has no neighbors in the set of unvisited vertices $U$ and the existence of such edges is unconditioned at this point. Thus w.h.p. 
\beq{b0}{
|U|=\frac{2\log n}{p_1}\text{ and }j_0=n-\frac{2\log n}{p_1}.
}
Next let $j_1=\min\set{j:j\in U}$. Then $\Pr(j_1\leq k)\leq n\E_{j_0}(1-p_1)^{j_0-k}$. This is because $j_1\leq k$ implies that $j_0-k$ non-edges have been reported for vertex $j_1$. So, w.h.p.,
\beq{b1}{
j_1\geq n-\frac{4\log n}{p_1}.
}
Now let $\a_j=|\set{k>j:v_k<v_j}|$ for all $1\leq j\leq n$, so that $\i(H)=\a_1+\a_2+\cdots+\a_{n}$.  If we can complete $(v_1,v_2,\ldots,v_{j_0})$ to a Hamilton cycle $H$, then
\[
\i(H)\leq \a_1+\a_2+\cdots+\a_{j_0}+|U|(n-j_1).
\]
Next we define an approximation $a_j$ to $\a_j$. We let $a_j=|\set{t<v_j:t\notin V_j}|$ for all $j\geq 1$, where $V_j=\set{v_1,v_2,\dots,v_j}.$  Observe that $\a_j\leq a_j$ for $j\leq j_0$.  Moreover,
\beq{aj+}{
\Pr(a_j=k)=(1-p_1)^kp_1\quad\mbox{for }k\geq 0.
}
To see this, observe that the vertex $v_j$ was chosen as the leftmost vertex available to the algorithm at round $j$, and determining this vertex involves querying edges which have not yet been conditioned by the running of the algorithm.  Observe that \eqref{aj+} holds even when conditioning on any previous history of the algorithm.  

So $a_1=0$ and $a_2,a_3,\dots$ is a sequence of independent copies of $Geo(p_1)-1$ where $Geo(p_1)$ is the geometric random variable with probability of success $p_1$.  We thus have:
\beq{geoj}{
\E\left(\sum_{j=0}^{j_0}\a_j\right) \leq \E\left(\sum_{j=0}^{j_0}a_j\right)\leq  \E\left(\sum_{j=0}^{n}a_j\right)\leq n\frac{1-p_1}{p_1}.
}
Moreover, standard concentration arguments give that $\sum_{j=0}^{j_0}\a_j \leq 2n/p_1$ w.h.p. So, if we can complete $(v_1,v_2,\ldots,v_{j_0})$ to a Hamilton cycle $H$, then w.h.p.
\beq{geojqs}{
\i(H)\leq \frac{2n}{p_1}+\frac{16\log^2n}{p_1^2}\leq M,
}
given that
\[
p\geq \frac{100\max\set{K,1}n}{M}.
\]
All that remains it to show that using the edges of $G_2$, we can w.h.p. extend $(v_1,v_2,\ldots,v_{j_0})$ to a Hamilton cycle. For this, we only have to show that there is a Hamilton path in the sub-graph $\G$ of $G_2$ induced by $U$ that can be added to $P_0$ to create a Hamilton cycle through $[n]$.

Let $N=\frac{2\log n}{p_1}$ and observe that $p_2\geq \frac{4\log N}{3N}$. Indeed,
\[
\frac{4\log N}{3Np_2}=\frac{4(\log2+\log\log n+\log 1/p_1)p_1}{3p_2\log n}\lesssim \frac23.
\]
It follows from standard results (see Chapter 6 of \cite{FrKa}) that there is a positive constant $c>0$ such that w.h.p. there are in $\G$, $cN$ vertices $x_1,x_2,\ldots,x_{cN}$ such that for each $i$ there are $cN$ Hamilton paths with one endpoint $x_i$ and otherwise distinct endpoints. So the probability we cannot add a Hamilton path in $\G$ to $P_0$ is at most $2(1-p_2)^{cN}=o(1)$. This completes the proof of Theorem \ref{greedy}.
\section{Comments and open problems}
While Theorems \ref{th1} -- \ref{th2a} are fairly general they can be improved in at least two ways. First we can ask for hitting time versions where we wait for sufficiently many edges and colors. Second and more challenging would be to prove that our random graphs {\em simultaneously} contain all posssible sequences, rather than a specific one.

In the case of Theorem \ref{th3} the bound $s_0=o(\log\log\log n)$ should probably be replaced by $s_0=o(\log\log n)$ in line with the fact that most pairs of vertices in $G_{n,p},p\sim \log n/n$ are $O(\log n/\log\log n)$ apart.

\end{document}